\documentclass[11pt]{article}
\usepackage{amsmath, amssymb, amsthm, amsfonts}
\usepackage{authblk}
\usepackage{fixltx2e}
\usepackage[margin=1.3in]{geometry}

\theoremstyle{thmit} 
\newtheorem{thm}{Theorem}[section]
\newtheorem{lem}[thm]{Lemma}
\newtheorem{cor}[thm]{Corollary}
\newtheorem{prop}[thm]{Proposition}

\theoremstyle{thmrm} 

\newtheorem*{rem}{Remark}
\newtheorem*{oldproof}{Proof}
\renewenvironment{proof}[1][{}]{\begin{oldproof}[#1]}{\qed\end{oldproof}}

\usepackage{amsmath, amssymb, amsthm, amsfonts}

\usepackage{amssymb}
\usepackage{color}
\usepackage{amsmath}
\usepackage{amsthm}
\usepackage{comment}

\newcommand{\mbb}{\mathbb}
\newcommand{\N}{\mathbb N}
\newcommand{\Z}{\mathbb Z}
\newcommand{\R}{\mathbb R}

\newcommand{\Q}{\mathbb Q}
\newcommand{\Homeo}{\rm Homeo}
\let\ol=\overline
\let\ap=\alpha

\newcommand{\T}{\overline{T}}
\newcommand{\mT}{\mbb{T}}
\newcommand{\Sc}{\mathcal{S}}
\newcommand{\ms}{\mathfrak{S}}
\newcommand{\Id}{\rm Id}
\newcommand{\U}{\mathcal U}
\newcommand{\D}{\mathcal D}
\newcommand{\G}{\mathcal G}
\theoremstyle{remark}
\theoremstyle{definition}

\title{Distality of Certain Actions on $p$-adic Spheres}
\author{Riddhi Shah}
\author{Alok Kumar Yadav}
\affil{Jawaharlal Nehru University, New Delhi, India}

\begin{document}

\maketitle
\begin{abstract}
Consider the action of $GL(n,\Q_p)$ on the $p$-adic unit sphere $\Sc_n$ arising from the linear action on $\Q_p^n\setminus\{0\}$. We show that for the action of a semigroup 
$\mathfrak{S}$ of $GL(n,\Q_p)$ 
on $\Sc_n$, the following are equivalent:  (1) $\mathfrak{S}$ acts distally on $\Sc_n$. (2) the closure of the image of $\mathfrak{S}$ in $PGL(n,\Q_p)$ is a compact group. 
On $\Sc_n$,  
we consider the `affine' maps $\T_a$ corresponding to $T$ in $GL(n,\Q_p)$ and a nonzero $a$ in $\Q_p^n$ satisfying $\|T^{-1}(a)\|_p<1$. We show that there exists a compact open 
subgroup $V$, 
which depends on $T$,  such that $\T_a$ is distal for every nonzero $a\in V$ if and only if $T$ acts distally on $\Sc_n$. The dynamics of `affine' maps on $p$-adic unit spheres is 
quite different from 
that on the real unit spheres. 



\end{abstract}


\section{Introduction}\label{s:1}
Let $X$ be a (Hausdorff) topological space. Let $\mathfrak{S}$ be a semigroup of homeomorphisms of $X$. The action of $\ms$ is said to be {\it distal} if for any pair of distinct 
elements $x, y\in X$, 
the closure of $\{(T(x),T(y))\mid T\in\mathfrak{S}\}$ does not intersect the diagonal $\{(d,d)\mid d\in X \}$; (equivalently we say that the $\ms$ acts {\it distally} on $X$). 
Let $T:{X}\to{X}$ be a homeomorphism. The map $T$ is said to be {\it distal} if the group 
$\{T^n\}_{n\in\Z}$ acts distally on $X$. If $X$ is compact, then $T$ is distal if and only if the semigroup $\{T^n\}_{n\in\N}$ acts distally (cf.\ Berglund et al.\ \cite{BJM}). 

The notion of distality was introduced by Hilbert (cf.\ Moore \cite{M9}) and studied by many in different contexts (see Ellis \cite{E4}, 
Furstenberg \cite{F6}, Raja-Shah \cite{RaSh10,RaSh11} and Shah \cite{Sh12}, and references cited therein). Note that a homeomorphism $T$ of a 
topological space is distal if and only if $T^n$ is so, for every $n\in\Z$.  

For the $p$-adic field $\Q_p$, let $|\cdot|_p$ denote the $p$-adic absolute value on $\Q_p$ and for $x=(x_1,\ldots, x_n)\in\Q_p^n$, $n\in\N$, let 
$\|x\|_p=\max_{1\leq i\leq n}|x_i|_p$. This defines a $p$-adic vector space norm on $\Q_p^n$. 
Let $\mathcal{S}_n=\{x\in\Q_p^n\mid \|x\|_p=1\}$ be the $p$-adic unit sphere (in $\Q_p^n$). We refer the reader to Koblitz \cite{K} for basic facts on $p$-adic vector spaces.  
We first define a canonical group action of $GL(n,\Q_p)$ on $\Sc_n$ as follows: For 
$T\in GL(n,\Q_p)$, let $\T:\Sc_n\to\Sc_n$ be defined as $\T(x)=\|T(x)\|_pT(x)$, $x\in\Sc_n$. This is a continuous group action. 
 We show that a semigroup $\ms$ of $GL(n,\Q_p)$ acts distally on $\Sc_n$ if and only if  
the closure of the image of $\ms$ in $PGL(n,\Q_p)=GL(n,\Q_p)/\D$ is a compact group; where $\D$ is the centre of $GL(n,\Q_p)$ (see Theorem \ref{cor-distal}). 
This is a $p$-adic analogue of Theorem~1 in Shah-Yadav \cite{SY} for the 
action of a semigroup in $GL(n+1,\R)$ on the real unit sphere $\mbb{S}^n$.   In particular, we show for 
$T\in SL(n,\Q_p)$ that $T$ is distal if and only if $\T$ is distal (more generally see Proposition \ref{act-Qp}, Corollary \ref{c} and the subsequent remark). 
For $T\in GL(n, \Q_p)$ and $a\in \Q_p^n\setminus\{0\}$, if $\|T^{-1}(a)\|_p<1$, then the corresponding `affine' map on $\mathcal{S}_n$,
$\overline{T}_a(x)=\|a+T(x)\|_p(a+T(x))$, $x\in \mathcal{S}_n$ is a homeomorphism. In case $T$, or more generally,  $\T$ is distal, then there exists a 
neigbourhood $V$ of 0 in $\Q_p^n$, such that for every nonzero $a$ in $V$, $\overline{T}_a$ is distal. If $\T$ is not distal, then every neighbourhood of 0 contains a 
nonzero $a$ such that $\T_a$ is not distal; see Theorem \ref{bar} and Corollary \ref{cbar}. In case of such `affine' actions on the real unit sphere $\mbb{S}^n$, there are many 
examples where $T\in GL(n+1,\R)$ such that
$T$ and/or $\T$ are distal but $\T_a$ is not distal, (see Theorem~7, Corollaries 10 \& 12 in \cite{SY}).   This illustrates that the dynamics of such `affine' actions on 
$\mathcal{S}_n$ is different from that on $\mbb{S}^n\subset\R^{n+1}$. 

For an invertible linear map $T$ on a $p$-adic vector space $V\approx\Q_p^n$, let $C(T)=\{v\in V\mid T^m(v)\to 0\mbox{ as }m\to\infty\}$ and 
$M(T)=\{v\in V\mid \{T^m(v)\}_{m\in\Z}\mbox{ is relatively compact}\}$. These are closed subspaces of $V$, $C(T)$ is known as the contraction space of $T$, and 
$V=C(T)\oplus  M(T)\oplus C(T^{-1})$.  It is easy to see that $T$ is distal (on $V$) if 
and only if $C(T)$ and $C(T^{-1})$ are trivial. We refer the reader to Wang \cite{W} for more details on the structure of $p$-adic contraction spaces. We will use the notion of 
contraction spaces below.


\section{Distality of the semigroup actions on $\mathcal{S}_n$}\label{s:2}

Let $\Q_p^n$ be an $n$-dimensional $p$-adic vector space equipped with the $p$-adic norm defined as above. 
For $T\in GL(n,\Q_p)$, let $\|T\|_p=\sup\{\|T(x)\|_p\mid x\in\Q_p,\, \|x\|_p=1\}$. Observe that the norm of an element or a matrix, defined this way, is of the from 
$p^m$ for some $m\in\Z$. The map $GL(n,\Q_p)\times \Q_p^n\to\Q_p^n$ given by $(T,x)\mapsto T(x)$, $T\in GL(n,\Q_p)$, $x\in \Q_p^n$, is continuous. 
We call $T\in GL(n,\Q_p)$ an {\it isometry} if it preserves the norm, i.e. if $T$ keeps the $p$-adic unit sphere $\Sc_n$ invariant. 
Note that $T$ is an isometry if and only if $\|T\|_p=1=\|T^{-1}\|_p$. For $x,y\in \Q_p^n$, 
$\|x+y\|_p\leq\max\{\|x\|_p,\|y\|_p\}$; the equality holds if $\|x\|_p\ne\|y\|_p$. We will use this fact extensively. 
We first consider the group action of $GL(n,\Q_p)$ on $\Sc_n$ defined in the introduction. For semigroups of $GL(n,\Q_p)$, we prove a result analogous to Theorem~1 in \cite{SY} 
(see Theorem~\ref{cor-distal}). 

Recall that $T\in GL(n,\Q_p)$ is said to be distal if $\{T^m\}_{m\in\Z}$ acts distally on $\Q_p^n$. 

The following useful lemma may be known. We will give a short proof for the sake of completion. 

\begin{lem} \label{distal-cpt} Let $T\in GL(n,\Q_p)$. The following are equivalent.
\begin{enumerate}
\item[$(1)$] $T$ is distal.
\item[$(2)$] The closure of the group generated by $T$ in $GL(n,\Q_p)$ is compact.
\item[$(3)$] $T^m$ is an isometry for some $m\in\N$.
\end{enumerate}
\end{lem}
\begin{proof}
$(3)\Rightarrow(2)$ is obvious and $(2)\Rightarrow (1)$  follows as compact groups act distally. Now suppose $T$ is distal, i.e.\ $\{T^m\}_{m\in\Z}$ acts distally on $\Q_p^n$. 
Then the contraction spaces $C(T)$ and $C(T^{-1})$ are trivial. By Lemma 3.4 of \cite{W}, we get that
$\Q_p^n=M(T)=\{x\in\Q_p^n\mid\{T^m(x)\}_{m\in\Z}\mbox{ is relatively compact}\}$, (cf.\ \cite{W}). By Proposition 1.3 of \cite{W}, $\cup_{m\in\Z}T^m(\Sc_n)$ is 
relatively compact, i.e.\ 
$\{\|T^m\|_p\}_{m\in\Z}$ is bounded and hence $\{T^m\mid m\in\Z\}$ is 
relatively compact in $GL(n,\Q_p)$. This proves $(1)\Rightarrow (2)$. Now suppose $T$ is contained in a compact group. Then $T^{\pm m_k}\to \Id$, 
for some $\{m_k\}\subset\N$. Therefore, $\|T^{\pm m_k}\|_p\to 1$ and as $\{\|T^m\|_p\mid m\in\Z\}\subset \{p^l\mid l\in\Z\}$, we get that for all large $k$, 
$\|T^{\pm m_k}\|_p=1$ and $T^{m_k}$ is an isometry. Therefore, $(2)\Rightarrow (3)$.
\end{proof}

We say that a topological group $G$ acts continuously on a topological space $X$ by homeomorphisms if there is a homomorphism 
$\psi: G \to\Homeo(X)$ such that the corresponding map $G\times X\to X$ given by $(g,x)\mapsto \psi(g)(x)$, $g\in\ G$, $x\in X$, is continuous. We say that a semigroup $H$ of $G$
(resp.\ $T\in G$) acts distally on $X$ if $\psi(H)$ acts distally on $X$ (resp.\ $\psi(T)$ is distal). We state a useful lemma which is well-known and can be proven easily.

\begin{lem} \label{lem-cpt} Let $X$ be a Hausdorff topological space and let $G$ be a Hausdorff topological group which acts continuously on $X$ by 
homeomorphisms. Let $H$ be a semigroup and $K$ be a compact subgroup of $G$ such that all the 
elements of $H$ normalise $K$. Then the semigroup $HK$ acts distally on $X$ if and only if $H$ acts distally on $X$. In particular, 
if $T,S\in G$ are such that $TS=ST$ and $S$ generates a relatively compact group in $G$, then $T$ acts distally on $X$ if and only if $TS$ acts distally on $X$. 
\end{lem}

We now recall the natural action of $GL(n,\Q_p)$ on $\Sc_n$ defined earlier: For $T\in GL(n,\Q_p)$ and $x\in \Sc_n$, $\T(x)=\|T(x)\|_pT(x)$. Here, $\T$ defines a 
homeomorphism of $\Sc_n$ and it is trivial if and only if $T=p^n\Id$ for some $n\in\Z$.
The map $GL(n,\Q_p)\to\Homeo(\Sc_n)$, given by $T\mapsto\T$ as above, is a homomorphism which factors through the discrete 
central subgroup $\{p^n\Id\mid n\in\Z\}$ of $GL(n,\Q_p)$. The corresponding map $GL(n,\Q_p)\times \Sc_n\to \Sc_n$, given by $(T,x)\mapsto \T(x)$, is continuous.  Therefore, 
$GL(n,\Q_p)$ acts continuously on $\Sc_n$ by homeomorphisms as above. Observe that $\mathcal{S}_1=\Z_p^*=\{x\in\Q_p\mid |x|_p=1\}$ and 
$GL(1,\Q_p)=\Q_p\setminus\{0\}$ acts distally 
on $\Sc_1$ as $\T=\|T\|_pT\in \Sc_1$ for every $T\in GL(1,\Q_p)$. The following will be useful in proving the main result of this section. 

\begin{prop} \label{act-Qp}
Let $T\in GL(n,\Q_p)$. If $bT$ is distal for some $b\in\Q_p$, then $\T$ is distal. Conversely, if $\T$ is distal, then for some $m\in\N$ and $l\in\Z$, $p^lT^m$ is distal. 
If $|\det T|_p=1$ and $\T$ is distal, then $T$ is distal.\end{prop}

\begin{proof}
Observe that $bT$ is distal if and only if $|b|_p^{-1}T$ is so. As $\T=\ol{p^mT}$ for any $m\in\Z$, we may replace $T$ by $|b|_p^{-1}T$ and assume that $T$ is distal. By 
Lemma \ref{distal-cpt}, $T$ generates a relatively compact group, and hence $\T$ is distal. Conversely, suppose $\T$ is distal. By 3.3 of \cite{W}, there exists $m\in\N$ such that 
$T^m=AUC$, where $C$ is a diagonal matrix, $U$ is unipotent, $A$ is semisimple, $A$, $U$ and $C$ commute with each other and $A$ as well as $U$ generate a relatively 
compact group. Now by Lemma \ref{lem-cpt}, we have that $\overline{C}$ is distal. Here, $C=D D'=D'D$ for some diagonal matrices $D$ and $D'$ such that the diagonal entries 
of $D$ (resp.\ $D'$) 
are of the form $p^{l_k}$, $l_k\in\Z$, $k=1,\ldots n$ (resp.\ in $\Z_p^*$). Since $D'$ also generates a relatively compact group and it commutes with $D$, by 
Lemma \ref{lem-cpt}, $\overline{D}$ is distal.
It is enough to show that $D=p^l\Id$, as in this case, $D$ would be central in $GL(n,\Q_p)$ and 
this would imply that $AU$ and $D'$ commute, and hence, $p^{-l}T^m=AUD'$ would generate a relatively compact group which in turn would imply that it is distal. If possible, 
suppose $p^l$ and $p^{l_1}$ are two entries in $D$ such that 
$l<l_1$. As $\ol{D}=\ol{p^{-l}D}$, we get for $D_1=p^{-l}D$ that $\ol{D}_1=\ol{D}$ is distal, $1$ is an eigenvalue of $D_1$ and $D_1$ has another eigenvalue 
$p^{l_1-l}$ which has $p$-adic absolute value less than 1. Then the contraction space of $D_1$, $C(D_1)\ne\{0\}$ as we can take a nonzero $y\in\Q_p^n$ satisfying $D_1(y)=p^ky$ 
for $k=l_1-l\in\N$; and it follows that $y\in C(D_1)$. 
Let $x\in\Sc_n$ be such that $D_1(x)=x$ and let $y$ be as above such that $0<\|y\|_p<1$. Then $\ol{D}_1(x)=x$ and $x+y\in \Sc_n$. Now $D_1^i(x+y)=(x+p^{ki}y)\to x\in \Sc_n$ as 
$i\to\infty$. Therefore, $\ol{D}_1^i(x+y)\to x$ and it leads to a contradiction as $\ol{D}_1$ is distal.  Therefore, $D=p^l\Id$ and $p^{-l}T^m$ is distal. 

Suppose $|\det T|_p=1$. Then $|\det(T^m)|_p=|\det T|_p^m=1$. As $\T$ is distal, $T^m=p^l S$ for some $l\in\Z$, where $S$ generates a relatively 
compact group. Therefore, $|\det S|_p=1$, and hence $l=0$ and $T^m=S$. This implies that $T$ also generates a relatively compact group, and by Lemma \ref{distal-cpt}, it is distal.
\end{proof}

From now on, ${\D}=\{b \, \Id\mid b\in{\Q_p}\}$, the centre of $GL(n,\Q_p)$.  The following characterises distal actions of semigroups on $\Sc_n$. 
Recall that $PGL(n,\Q_p)=GL(n,{\Q_p})/{\D}$. 

\begin{thm} \label{cor-distal}
Let $\mathfrak{S}\subset GL(n,\Q_p)$ be a semigroup. Then the following are equivalent:
\begin{enumerate}
\item[$(i)$] $\mathfrak{S}$ acts distally on $\Sc_n$.
\item[$(ii)$] The group generated by $\ms$ acts distally on $\Sc_n$. 
\item[$(iii)$] The closure of $(\mathfrak{S}\D)/\D$ in $PGL(n,\Q_p)$ is a compact group. 
\end{enumerate}
\end{thm}

\begin{proof} Suppose (i) holds. First suppose $\ms\subset SL(n,\Q_p)$. As the closure $\ol{\ms}$ of $\ms$ is a semigroup 
in $SL(n,\Q_p)$ and it also acts distally on $\Sc_n$, we may assume that $\ms$ is closed. By Proposition \ref{act-Qp} and Lemma \ref{distal-cpt}, each element in 
$\ms$ generates a relatively compact group (in $\ms$). In particular, each element of $\ms$  has an inverse in $\ms$ and $\ms$ is a group. Now by Lemma 3.3 of \cite{GR}, 
$\ms$ is contained in a compact extension of a unipotent subgroup in $GL(n,\Q_p)$ which is normalised by $\ms$. Now suppose 
$\ms\not\subset SL(n,\Q_p)$. We will first show that $\ms$ is contained in a compact extension of a nilpotent group in $GL(n,\Q_p)$ and the latter is isomorphic to a direct 
product of $\D$ and a unipotent subgroup normalised by elements of $\ms$.

 Let ${\mathcal C}=\{p^n\Id\mid n\in\Z\}$ and let ${\mathcal Z}=\{z\, \Id\mid z\in\Z_p^*\}$. Then ${\mathcal C}$ and 
${\mathcal Z}$ are closed subgroups of $\D$, ${\mathcal C}$ is discrete, ${\mathcal Z}$ is compact and $\D={\mathcal C}\times{\mathcal Z}$. As the actions of both $\ms$ and 
$\ms {\mathcal C}$ on $\Sc_n$ are 
same and the latter is also a semigroup, without loss of any generality, we may replace $\ms$ by $\ms {\mathcal C}$ and assume that
${\mathcal C}\subset\ms$. As noted earlier, we may also assume that $\ms$ is closed. Moreover, as ${\mathcal Z}$ is compact and central, $\ms{\mathcal Z}$ is a closed
semigroup and by Lemma \ref{lem-cpt}, $\ms{\mathcal Z}$ acts 
distally on ${\Sc_n}$. Therefore, we may replace $\ms$ by $\ms{\mathcal Z}$ and assume that ${\mathcal Z}\subset \ms$. Now we have $\D\subset\ms$ and $\ms\D=\ms$.  
Let $T\in \ms$. By Proposition \ref{act-Qp}, $T^m=p^lS$ for some $l\in\Z$ and $S\in GL(n,\Q_p)$ such that  
$S$ generates a relatively compact group. Moreover, as $\D\subset\ms$, we have that $S\in\ms$. Therefore, the closure of the semigroup generated by $S$ in $\ms$ is compact and 
hence a group. In particular, $S$ is invertible in $\ms$ and we have that $T^m$ is invertible in $\ms$,
 and hence $T^{m-1}(T^m)^{-1}$ is the inverse of $T$ in $\ms$. Therefore, we may assume that $\ms$ is a closed  group. Let $\mT$ be a maximal torus in $GL(n,\Q_p)$. 
 Let $\mT_a$ (resp.\ $\mT_d$) be the anisotropic (resp.\ split) torus in $\mT$. Then $\mT_a$ is compact and $\mT_a\mT_d$ is an almost direct product. Moreover, 
 since all maximal tori are conjugate to each other, there 
 exists $m\in\N$, such that for any element $T\in\ms\subset GL(n,\Q_p)$, we have $T^m=\tau_a \tau_d \tau_u$, where $\tau_u$ is unipotent, 
 $\tau_s=\tau_a\tau_d\in \mT$ is semisimple, $\tau_a\in \mT_a$ which is a compact group, $\tau_d\in \mT_d$ and $\tau_a$, $\tau_d$ and $\tau_u$ commute with each other. 
 Note that $\mT$ depends on $T$, but $m$ is independent of the choice of $T$. We know that $\tau_u$, being unipotent, generates a relatively compact group. As $\T$ is distal, 
 arguing as above in the 
 proof of Proposition \ref{act-Qp}, we get that $\tau_d=(t_{ij})_{n\times n}$ is such that $t_{ij}=0$ if $i\ne j$ and $|t_{ii}|_p$ is the same for all $i$. We have $T^m=CS$, where $C\in \D$ 
 and $S$ generates a relatively 
 compact group. Let $\pi: GL(n,\Q_p)\to GL(n,\Q_p)/\D$ be the natural projection. Note that $GL(n,\Q_p)/\D$ is an algebraic group and it is linear, i.e.\ it can be realised as a 
 subgroup of 
 $GL(V)$ for some $p$-adic vector space $V$. Now $\pi(\ms)$ is a (closed) subgroup of $GL(n,\Q_p)/\D$ and every element of $\pi(\ms)$ generates a relatively compact 
 group, by Lemma 3.3 of \cite{GR}, 
 $\pi(\ms)$ is contained in a compact extension of a unipotent group, i.e.\ $\pi(\ms)\subset K\ltimes\, \U$, a semi-direct product, where $K, \, \U\subset GL(V)$, $K$ is a 
 compact group and $\U$ is unipotent. 
 We can choose $K$ such that $\ol{\pi(\ms)\, \U}/\U$ is isomorphic to $K$. Let $H=K\ltimes \U=\ol{\pi(\ms)\, \U}$. 
 Let $\G$ be the smallest algebraic subgroup of $GL(n,\Q_p)/\D$ containing $\pi(\ms)$. Here, since $\U$ is unipotent and it is normalised by $\pi(\ms)$, it is also 
 normalised by $\G$.  Then $\G\, \U$ is an algebraic group, and hence it is closed. As $H\subset \G\, \U$ and 
  $H=(\G\cap H)\,\U$, it follows that $K=H/\U$ is isomorphic to $(\G\cap H)/(\G\cap \U)$. Let $H_0=\G\cap H$ and let $\U_0=\G\cap \U$. Then 
  $\pi(\ms)\subset H_0$ and $\pi(\ms)$ normalises $\U_0$. Here, $\pi^{-1}(\U_0)=\D\times\U'$ where $\U'$ is the unipotent radical of 
  $\pi^{-1}(\U_0)$, $\U'$ is isomorphic to $\pi(\U')=\U_0$ and it is normalised by $\ms$.
 Therefore $\ms\subset H'=\pi^{-1}(H_0)$. Also, $\D\times\,\U'$ is a co-compact 
 nilpotent normal subgroup in $H'$ and it is also algebraic,  (see \cite{BoT} and \cite{PR} for details on algebraic groups). 
  
 By Kolchin's Theorem, there exists a flag $\{0\}=V_0\subset\cdots\subset V_k=\Q_p^n$ of maximal $\,\U'$-invariant subspaces such that $\U'$ acts trivially on 
$V_j/V_{j-1}$,  $j=1,\ldots, k$. Note that each $V_j$ is maximal in the sense that for any subspace $W$ containing $V_j$ such that $W\ne V_j$, 
$\,\U'$ does not act trivially on $W/V_{j-1}$. It is easy to see that each $V_j$ is $\ms$-invariant as $\ms$ normalises $\,\U'$. 
Now suppose $\ms/\D$ is not compact. Then there exists a sequence $\{T_i\}\subset\ms$ such that 
$\{\pi(T_i)\}$ is unbounded. Now we have $T_i=K_iD_iU_i$, $i\in\N$, where $D_i\in\D$, $U_i\in \U'$ and $K_i\in \pi^{-1}(K)$ such that $\{K_i\}$ is relatively compact and 
$\{U_i\}$ is unbounded. Note that for each $j$, as $\ms$, $\U'$ and $\D$ keep $V_j$  invariant, $K_i(V_j)=V_j$ for all $i$. 
 Passing to a subsequence if necessary, we get that there exists $w\in \Sc_n$ such that $\|U_i(w)\|_p\to \infty$,  $\T_i(w)\to w'\in \Sc_n$ and that $K_i\to K_0$. For every $v\in V_1$,
$T_i(v)=D_iK_i(v)\in V_1$, and hence $\{\|D_i^{-1}T_i(v)\|_p\}$ is bounded. Let $v\in V_1\setminus\{0\}$ be such that $\|v\|_p<1$. Then $v+w\in \Sc_n$.
As $\{\|K_iU_i(v+w)\|_p\}$ is unbounded and $K_iU_i(v)=K_i(v)\to K_0v$, we get that $\T_i(v+w)=\ol{K_iU_i}(v+w)\to w'$.
 This contradicts (i). Therefore, $\ms/\D$ is compact, i.e.\ $(i)\Rightarrow(iii)$. 

Suppose $\ol{\ms \D}/\D$ is a compact group. Then it contains $G\D/\D$, where $G$ is the group generated by $\ms$. Therefore, $G\D/\D$ is relatively compact. Using this fact, 
we want to show that $G$ acts distally on $\Sc_n$. Let ${\mathcal C}$ and ${\mathcal Z}$ be closed subgroups of $\D$ as above. Since the actions of both $G$ and 
$G{\mathcal C}$ on $\Sc_n$ are 
same, without loss of any generality we may replace $G$ by $G{\mathcal C}$ and assume that ${\mathcal C}\subset G$. We may also assume that $G$ is closed. Now, using 
the facts that $\D={\mathcal C}\times{\mathcal Z}$ and ${\mathcal Z}$ is compact, we get that $G\D=G{\mathcal Z}$ is closed. Now $G\D/\D$ is compact and it is isomorphic to 
$G/[(G\cap{\mathcal Z})\times{\mathcal C}]$. Therefore, $G/{\mathcal C}$ is compact, as $G\cap {\mathcal Z}$ is compact. Since the action of $G$ on $\Sc_n$ factors through 
${\mathcal C}$, the preceding assertion implies that $G$ acts 
distally on $\Sc_n$. Hence $(iii)\Rightarrow (ii)$. It is obvious that $(ii)\Rightarrow(i)$. 
\end{proof}

The following is a consequence of the above theorem.

\begin{cor} \label{c}
A semigroup $\ms$ of $SL(n,\Q_p)$ acts distally on 
$\Sc_n$ if and only if the closure of $\ms$ is a compact group. 
\end{cor}

\begin{proof} Let $\pi: GL(n,\Q_p)\to GL(n,\Q_p)/\D$ be as above. By Theorem \ref{cor-distal}, $\ms$ acts distally on $\Sc_n$ if and only if $\ol{\pi(\ms)}$ is a compact group. Note that 
$\pi(\ol{\ms})\subset \ol{\pi(\ms)}$. As $SL(n,\Q_p)\D$ is closed and $SL(n,\Q_p)\cap\D$ is finite, we get that $\ol{\pi(\ms)}$ is compact 
if and only if $\ol{\ms}$ is compact; the latter statement is equivalent to the following: $\ol{\ms}$ is a compact group. Now the assertion follows from above. 
\end{proof}

\begin{rem} Note that the above corollary is valid for a semigroup $\ms\subset GL(n,\Q_p)$ satisfying the condition that $|\det(T)|_p=1$ for all $T\in\ms$, as the elements 
of $GL(n,\Q_p)$ satisfying the condition form a closed subgroup (say) $G$ such that $G\D$ is closed and $G\cap\D$ is a compact group (isomorphic to $\Z_p^*$).
\end{rem} 

In the real case, Corollary 5 in \cite{SY} showed that a semigroup $\ms\subset GL(n+1,\R)$ acts distally on the (real) unit sphere $\mbb{S}^n$ if (and only if) 
every cyclic subsemigroup of $\ms$ acts distally on $\mbb{S}^n$. The corresponding statement does not hold in the $p$-adic case, as there exists a class of 
closed non-compact subgroups of $SL(n,\Q_p)$, every cyclic subgroup of which 
is relatively compact but which do not act distally on $\Sc_n$ as they are not compact; e.g.\ the group of strictly upper triangular matrices in $SL(n,\Q_p)$, $n\geq 2$. 

\section{Distality of `affine' actions on $\mathcal{S}_n$}

In this section, we discuss  the `affine' maps on the $p$-adic unit sphere $\Sc_n$. Consider the affine action on $\Q_p^n$; $T_a(x)=a+T(x)$, $x\in\Q_p^n$, where 
$T\in GL(n,\Q_p)$, and $a\in \Q_p^n$. We 
first consider the corresponding `affine' map $\overline{T}_a$ on $\Sc_n$ which is defined for any nonzero $a$ satisfying $\|T^{-1}(a)\|_p\ne 1$ as follows: 
$\overline{T}_a(x)=\|T_a(x)\|_pT_a(x)$, $x\in \Sc_n$. 
(For $a=0$, $\overline{T}_a=\overline{T}$, which is studied in Section 2). Observe that $T_a(x)=0$ for some $x\in \Sc_n$ if and only if $T^{-1}(a)$ has norm 1. Therefore, 
$\ol{T}_a(\Sc_n)\subset \Sc_n$ if $\|T^{-1}(a)\|_p\ne 1$. The map $\T_a$ is a homeomorphism for any nonzero $a$ satisfying 
$\|T^{-1}(a)\|_p<1$ (see Lemma~\ref{e} below). In this section, we study the distality of such homeomorphisms $\overline{T}_a$.

\begin{lem}\label{e}
Let $T\in GL(n, \Q_p)$ and let $a\in\Q_p^n\setminus\{0\}$ be such that $\|T^{-1}(a)\|_p\ne 1$. Then the map $\overline{T}_a$ on $\Sc_n$ is continuous and injective. $\ol{T}_a$ 
is a homeomorphism if and only 
if $\|T^{-1}(a)\|_p<1$.
\end{lem}

\begin{proof}
Suppose $\|T^{-1}(a)\|_p\ne 1$. From the definition of $\overline{T}_a$, it is obvious that it is continuous. Suppose $x,y\in \Sc_n$ such that $\overline{T}_a(x)=\overline{T}_a(y)$. 
Then $$\|a+T(x)\|_p\left(a+T(x)\right)=\|a+T(y)\|_p\left(a+T(y)\right)$$ 
or $(\beta-1)T^{-1}(a)=y-\beta x$, 
where $\beta=\|a+T(x)\|_p/\|a+T(y)\|_p=p^m$ for some $m\in\Z$. If possible, suppose $\beta\ne 1$. Interchanging $y$ and $x$ if necessary, we may assume that $\beta>1$ or 
equivalently, that $m\in\N$. 
This implies that $\|\beta x\|_p=|\beta|_p=p^{-m}<1$, and we get $$\|T^{-1}(a)\|_p=|\beta-1|_p\|T^{-1}(a)\|_p=\|y-\beta x\|_p=1,$$ 
a contradiction. Hence, $\beta=1$ and $x=y$. Therefore, $\overline{T}_a$ is injective.

Now suppose $\|T^{-1}(a)\|_p<1$. It is enough to show that $\T_a$ is surjective, as any continuous bijection 
on a compact Hausdorff space is a homeomorphism. 

Let $y\in\Sc_n$. Let $z=T^{-1}(y)$ and let $x=\|z\|_p z-T^{-1}(a)$.  Since the norm of $\|z\|_pz$ is $1$  and $\|T^{-1}(a)\|_p<1$, we have that $\|x\|_p=1$. Moreover, as $\|y\|_p=1$, 
we have that 
$\|z\|_p^{-1}=\|a+T(x)\|_p$. Therefore, $\T_a(x)=y$. Hence $\overline{T}_a$ is surjective. 

Conversely, Suppose $\T_a$ is surjective. Then there exists $x\in\Sc_n$ such that $\T_a(x)=\|a\|_pa$. We get that $x=(p^m-1)T^{-1}(a)$, where $p^m=\|a+T(x)\|_p^{-1}\|a\|_p$ 
for some $m\in\Z$; here 
$m\neq 0$ since $x\ne 0$. Now $1=\|x\|_p=|p^m-1|_p\|T^{-1}(a)\|_p\geq\|T^{-1}(a)\|_p$ since $|p^m-1|_p\geq 1$ for every $m\in \Z\setminus\{0\}$. As $\|T^{-1}(a)\|_p\ne 1$, we 
have that $\|T^{-1}(a)\|_p<1$.
\end{proof}

 In \cite{SY}, we have studied `affine' maps $\T_a$ on the real unit sphere $\mbb{S}^n$. The following result shows that in the 
 $p$-adic case, $\T_a$ is distal for every nonzero $a$ in a certain neighbourhood of 0 in $\Q_p^n$ if and only if $\T$ is distal. This illustrates  
 that the behaviour of such maps in the $p$-adic case is very different from that in the real case.

\begin{thm}\label{bar} 
Suppose $T\in GL(n,\Q_p)$. For $a\in \Q_p^n\setminus\{0\}$ with $\|T^{-1}(a)\|_p\ne 1$, let $\T_a: \mathcal{S}_n\to \mathcal{S}_n$ be defined as 
$\overline{T}_a(x)=\|a+T(x)\|_p(a+T(x))$, $x\in\Sc_n$. There exists a compact open subgroup $V\subset\Q_p^n$ 
such that for all 
$a\in V\setminus\{0\}$ we have $\|T^{-1}(a)\|_p<1$, $\overline{T}_a$ is a homeomorphism and the following hold:
\begin{enumerate}
\item[$({\rm I})$] If $\T$ is distal, then $\ol{T}_a$ is distal for all nonzero $a\in V$.
\item[$({\rm II})$] If $\T$ is not distal, then for every neighbourhood $U$ of 0 contained in $V$, there exists a nonzero $a\in U$ such that $\overline{T}_a$ is not distal.  
\end{enumerate}
\end{thm}

\begin{proof}
By 3.3 of \cite{W}, we get that there exist $D$ and $S$ which commute with $T$ and $m\in\N$ such that 
$T^m=SD=DS$, where $D$ is a diagonal matrix with the diagonal entries in $\{p^i\mid i\in\Z\}$ and $S$
generates a relatively compact group. Therefore, $S^k$ is an isometry for some $k\in\N$. Replacing $m$ by $km$, we may assume that $S$ itself is an isometry. Let 
$c_0=\min\{(1/\|T^{-j}\|_p)\mid 1\leq j\leq m-1\}$ and $c_1=\max\{\|T^j\|_p\mid 1\leq j\leq m-1\}$.  
As $\Sc_n$ is compact, $0<c_0\leq c_1<\infty$. Also, $c_0\leq \|T^j(x)\|_p\leq c_1$ for all $x\in \Sc_n$ and $1\leq j\leq m-1$. Since $\|T^j\|_p\in\{p^i\mid i\in\Z\}$, 
we get that $\{\|T^j(x)\|_p \mid x\in\Sc_n, 1\leq j\leq m-1\}$ is finite. 

Let $V$ be a compact open $S$-invariant subgroup in $\Q_p^n$ such that $V\cup c_0V\cup c_0^2V\cup c_0^2c_1^{-2}V\subset W=\{w\in \Q_p^n\mid \|w\|_p<1\}$. 
Then, $\|v\|_p<\min\{1,c_0,c_0^2,c_0^2c_1^{-2}\}\leq 1$ for all $v\in V$ and $c_0c_1^{-1}V\subset c_0^2c_1^{-2}V\subset W$. Moreover, $\|T^{-1}(v)\|_p<c_0^{-1}c_0=1$
for every $v\in V$.

 Let $p^l$ be the smallest nonzero entry in the diagonal matrix $D$ and let $H=\{x\in\Q_p^n\mid D(x)=p^lx\}$. 
This is a nontrivial closed subspace of $\Q_p^n$. As $S$ and $T$ commute with $D$, they keep $H$ invariant and, as $S$ is an isometry, $\|T^m(x)\|_p=p^{-l}\|x\|_p$ for all $x\in H$.

Let $a\in V\setminus\{0\}$. As noted above, $\|T^{-1}(a)\|_p<1$, and hence by Lemma \ref{e}, $\T_a$ is a homeomorphism. Take any $x\in \Sc_n$. 
Since $\|a\|_p<c_0$ and $\|T(x)\|_p\geq c_0$, we have $\|T_a(x)\|_p=\|a+T(x)\|_p=\|(T(x)\|_p$ and
$$\T_a(x)=\|T_a(x)\|_pT_a(x)=\|T(x)\|_p(a+T(x)).$$ 
Let $\ap_1(x)=\|T_a(x)\|_p=\|T(x)\|_p=\beta_{1,x}$. Let $\alpha_j(x)=\|T_a(\T_a^{j-1}(x))\|_p=\|a+T(\T_a^{j-1}(x))\|_p$ and 
let $\beta_{j,x}=\ap_1(x)\cdots\ap_j(x)$ for all $j\in\N$ ($j\geq 2$). Take $\beta_{0,x}=1$ and $\phi^0=\Id$ for any map $\phi$. From above, we have that 
$\ap_j(x)=\|T(\T_a^{j-1}(x))\|_p$ for all $j\in\N$. 
It is easy to show by induction that for every $j\in\N$,
\begin{equation}
\T_a^j(x)=\beta_{j,x}T^j(x)+\beta_{j,x}\sum_{i=1}^j\beta_{j-i,x}^{-1}T^{i-1}(a).
\end{equation}
Observe that as $a\in V$, $\|T^k(a)\|_p\leq c_1\|a\|_p<c_1(c_0c_1^{-1})=c_0$ and for any $x\in \Sc_n$, $\|T^k(x)\|_p\geq c_0$,  $1\leq k\leq m-1$. Therefore, for 
$j\in\N$ and $1\leq k\leq m-2$,
\begin{eqnarray*}
\|T^k(\T_a^j(x))\|_p&=&[\ap_j(x)]^{-1}\|T^k(a)+T^{k+1}(\T_a^{j-1}(x))\|_p\\
&=&[\ap_j(x)]^{-1}\|T^{k+1}(\T_a^{j-1}(x))\|_p.
\end{eqnarray*} 
Applying the above equation successively, we get that for $1\leq j\leq m-1$,
$\ap_j(x)=\|T(\T_a^{j-1}(x))\|_p=[\ap_{j-1}(x)\cdots\ap_1(x)]^{-1}\|T^j(x)\|_p$ i.e.\ $\beta_{j,x}=\|T^j(x)\|_p.$  
Hence, $c_0\leq \beta_{j,x}\leq c_1$ for all $x\in\Sc_n$ and $1\leq j\leq m-1$. 
Moreover, applying the same equation again successively, we get for $j\geq m$ that 
\begin{equation}
\ap_j(x)=[\ap_{j-1}(x)\cdots\ap_{j-m+1}(x)]^{-1}\|T^{m-1}(a)+T^m(\T_a^{j-m}(x))\|_p.
\end{equation}

Now we take $a\in V\cap H\setminus\{0\}$. Then $\T_a(\Sc_n\cap H)=\Sc_n\cap H$. Let $x\in \Sc_n\cap H$. Then $\|T^{-1}(a)\|_p<c_0^{-1}c_0=1$ and hence, 
$\|T^{-1}(a)+\T^{j-m}(x)\|_p=1$. This implies that 
\begin{equation}
\|T^{m-1}(a)+T^m(\T^{j-m}(x))\|_p=\|T^m(T^{-1}(a)+\T^{j-m}(x))\|_p=p^{-l}.
\end{equation}
Using Eqs.\ (2) and (3), we get 
 $\ap_j(x)=[\ap_{j-1}(x)\cdots\ap_{j-m+1}(x)]^{-1}p^{-l}$, and hence 
$\beta_{j,x}=p^{-l}\beta_{j-m,x}$ for all $j\geq m$. In particular, $\beta_{m,x}=p^{-l}=\|T^m(x)\|_p$. 
This implies that  
$\beta_{km+j,x}=p^{-kl}\beta_{j,x}=p^{-kl}\|T^j(x)\|_p$, $k,j\in\N$. Therefore, $\beta_{j,x}=\|T^j(x)\|_p$, $j\in\N$.  
Moreover, for all $j,k\in\Z$ and $x\in H$, $T^{km+j}(x)=p^{kl}S^kT^j(x)=p^{kl}T^jS^k(x)$ and, $\|T^{km+j}(x)\|_p=p^{-kl}\|T^j(x)\|_p$ as $S$ is an isometry. 
In particular, $\beta_{km+j,x}=p^{-kl}\|T^j(x)\|_p$ for all $k,j\in\Z$ such that $km+j\geq 0$.  Using the above facts together with Eq.\ (1), we get for $k\in\N$,
\begin{eqnarray*}
\T_a^{km}(x)&=&\beta_{km,x}T^{km}(x)+\beta_{km,x}\sum_{j=1}^{km}\beta_{km-j,x}^{-1}T^{j-1}(a)\\
&=&S^k(x)+\sum_{j=1}^{km}\|T^{-j}(x)\|_p^{-1}T^{j-1}(a)\\
&=&S^k(x)+\sum_{i=1}^k\sum_{j=1}^m\|T^{-j}(x)\|_p^{-1}T^{j-1}(S^{i-1}(a))\\
&=&S^k(x)+\sum_{j=1}^m\gamma_{j,x}^{-1}\,T^{j-1}(a_k),
\end{eqnarray*}
where $a_k=\sum_{i=1}^kS^{i-1}(a)\in V\cap H$, $k\in\N$, $\gamma_{j,x}=\|T^{-j}(x)\|_p=p^l\beta_{m-j,x}$ and $c_1^{-1}\leq\gamma_{j,x}\leq c_0^{-1}$, $1\leq j\leq m-1$, and 
$\gamma_{m,x}=p^l$. 
From above, we get that for any $k\in\N$ and $x,y\in \Sc_n\cap H$,
\begin{equation}
\T_a^{km}(x)-\T_a^{km}(y)=S^{k}(x-y)+\sum_{j=1}^{m-1}[\gamma_{j,x}^{-1}-\gamma_{j,y}^{-1}]\, T^{j-1}(a_k).\end{equation}
Let $x,y\in\Sc_n\cap H$ such that $\|x-y\|_p< c_0c_1^{-1}$.  As $T$ is linear, $T^j(x)=T^j(y)+T^j(x-y)$, $j\in\N$.  For $1\leq j\leq m-1$, 
as $\|T^j(x-y)\|_p\leq c_1\|x-y\|_p<c_0$, and $\|T^j(y)\|_p\geq c_0$, we get that $\beta_{j,x}=\|T^j(x)\|_p=\|T^j(y)\|_p=\beta_{j,y}$, and hence $\gamma_{j,x}=\gamma_{j,y}$. 
Therefore,  
$$\|\T_a^{km}(x)-\T_a^{km}(y)\|_p=\|S^k(x)-S^k(y)\|_p=\|x-y\|_p, \ k\in\N.$$ 
Now suppose $\|x-y\|_p\geq c_0c_1^{-1}$. 
Observe that $|\gamma_{j,x}^{-1}-\gamma_{j,y}^{-1}|_p\leq c_0^{-1}$, $\|T^j(a_k)\|_p\leq c_1\|a_k\|_p$, $1\leq j\leq m-1$ and $a_k\in V$, $\|a_k\|_p<c_0^2c_1^{-2}$. Now Eq.\ (4) 
implies that 
$\T_a^{km}(x)-\T_a^{km}(y) \in S^k(x-y)+c_0^{-1}c_1W$.  Since $\|S^k(x-y)\|_p=\|x-y\|_p\geq c_0c_1^{-1}$, we get that 
$\|\T_a^{km}(x)-\T_a^{km}(y)\|_p=\|x-y\|_p$. This shows that $\T_a^m|_{\Sc_n\cap H}$ preserves the distance and it is distal and hence, $\T_a|_{\Sc_n\cap H}$ is distal, 
where $a\in V\cap H$. 

If $\T$ is distal, then so is $\T^m$, and hence its image in $GL(n,\Q_p)/\D$ generates a relatively compact group. This implies that $D=p^l\Id$, $H=\Q_p^n$, $\Sc_n\cap H=\Sc_n$ 
and $V\cap H=V$. Therefore, (I) holds. 

Now suppose $\T$ is not distal. Then $\T^m$ is not distal and hence $D\ne p^l\Id$. Let $l_1>l$ and $H_1=\{x\in\Q_p^n\mid D(x)=p^{l_1}x\}$. Then $H_1$ is a vector 
subspace and it is invariant under $D$, $S$ and $T$. Let $a\in V\cap H\setminus\{0\}$ as above. It is easy to see that $\T_a(\Sc_n\cap (H\oplus H_1))=\Sc_n\cap (H\oplus H_1)$. 
 We show that the restriction of $\T_a$ to $\Sc_n\cap (H\oplus H_1)$ is not distal. This would imply that (II) holds. 

Take $y=x+z\in \Sc_n$, where $x\in\Sc_n\cap H$ and $z\in H_1$ such that $\|T^j(z)\|_p<\|T^j(x)\|_p$, $j\in\N$. It is possible to choose such a $z$; we can take $z\in H_1$ with the 
property that $\|T^j(z)\|_p<\|T^j(x)\|_p$
for all $0\leq j\leq m-1$, then as $S$ is an isometry, $\|T^{km+j}(z)\|_p=p^{-kl_1}\|T^j(z)\|_p<p^{-kl}\|T^j(x)\|_p=\|T^{km+j}(x)\|_p$, $k\in\N$. Now 
$\|T^j(y)\|_p=\|T^j(x)\|_p=\beta_{j,x}$ for all $j\in\N$. 
Here, 
$$\T^{km}(y)-\T^{km}(x)=p^{-kl}[S^k(p^{kl}x+p^{kl_1}z)]-S^k(x)=p^{k(l_1-l)}S^k(z)\to 0$$ 
as $k\to\infty,$ 
since $S$ is an isometry and since $l_1>l$. 
We now show for all $k\in\N$ that $\T_a^{km}(y)-\T_a^{km}(x)=\T^{km}(y)-\T^{km}(x)$. (From above, the latter is equal to $\beta_{km,x}T^{km}(z)$.) This in turn would imply that 
$\T_a$ is not distal. 

From Eq.\ (1), it is enough to show for all $j\in\N\cup\{0\}$ that $\beta_{j,y}=\beta_{j,x}$, or equivalently, $\beta_{j,y}=\|T^j(y)\|_p$ as the latter is equal to $\|T^j(x)\|_p$ which is the 
same as $\beta_{j,x}$. This is trivially true for $j=0$. 
As shown earlier, for $1\leq j< m-1$, $\beta_{j,u}=\|T^j(u)\|_p$ for all $u\in\Sc_n$, and hence $\beta_{j,y}=\beta_{j,x}$; i.e.\ the above statement holds for $1\leq j< m$, 
and we get that 
\begin{equation}
\T_a^j(y)=\beta_{j,y}T^j(y)+\beta_{j,y}\sum_{i=1}^j\beta_{j-i,y}^{-1}T^{i-1}(a)=\T_a^j(x)+\beta_{j,x}T^j(z).
\end{equation}
 We prove by induction on $k$ that $\beta_{j,y}=\beta_{j,x}=\|T^j(x)\|_p$ and Eq.\ (5) is satisfied for all $1\leq j< km$, $k\in\N$. We
have already proven these for $k=1$. Suppose for some $k\in\N$, these hold for all $j$ such that $(k-1)m\leq j< km$. Let $km\leq j< (k+1)m$.
Recall that for all $j\in\N$, $\ap_j(u)=\|T(\T_a^{j-1}(u))\|_p$, $u\in\Sc_n$, and Eq.\ (2) holds for any $x\in\Sc_n$ and $j\geq m$. 
As $\beta_{j,y}\beta_{j-m,y}^{-1}=\ap_j(y)\ldots\ap_{j-m+1}(y)$, 
from Eq.\ (2) and, also Eq.\ (5) which is assumed to hold for $(k-1)m\leq j<km$ by the 
induction hypothesis, we get for $x,y,z$ as above and $km\leq j< (k+1)m$ that 
\begin{eqnarray*}
\beta_{j,y}\beta_{j-m,y}^{-1}
&=&\|T^{m-1}(a)+T^m(\T_a^{j-m}(y))\|_p\\
&=&\|T^m[T^{-1}(a)+\T_a^{j-m}(x)+\beta_{j-m,x}T^{j-m}(z)]\|_p
\end{eqnarray*}
Now using this, we get that
\begin{equation*}
\beta_{j,y}\beta_{j-m,y}^{-1}=
\|S[p^l(T^{-1}(a)+\T_a^{j-m}(x))+p^{l_1}\beta_{j-m,x}T^{j-m}(z)]\|_p=p^{-l},
\end{equation*}
as $S$ is an isometry, $l_1>l$ and $\|\beta_{j-m,x}T^{j-m}(z)\|_p<1$ (see also Eq.\ (3)). Since $(k-1)m\leq j-m< km$, 
$\beta_{j,y}=p^{-l}\beta_{j-m,y}=\|p^lT^{j-m}(x)\|_p=\|T^j(x)\|_p$. Hence Eq.\ (5) holds for $km\leq j< (k+1)m$. Now by induction for all $j\in\N$, $\beta_{j,x}=\beta_{j,y}$ 
and Eq.\ (5) holds. 
Therefore, $\T_a$ is not distal. (Note that Eq.\ (5) also directly shows that 
$\T_a^{km}(y)-\T_a^{km}(x)=p^{k(l_1-l)}S^k(z)\to 0$ as $k\to\infty$). Now if $U\subset V$ is a neighbourhood of $0$, then $U\cap H\ne\{0\}$ and hence (II) holds.
\end{proof}

Observe that if $\T$ is not distal, then from Theorem \ref{bar} (II), we get that every neighbourhood of 0 in $\Q_p$ contains a nonzero $a$ such that 
$\|T^{-1}(a)\|_p<1$ and $\T_a$ is not distal. Now the following corollary is an easy consequence of Theorem \ref{bar}.  

\begin{cor} \label{cbar}
For $T\in GL(n,\Q_p)$, $\T$ is distal if and only if there exists a neighbourhood $V$ of 0 in $\Q_p^n$ such that for every $a\in V\setminus\{0\}$, $\|T^{-1}(a)\|_p<1$ and $\T_a$ on 
$\Sc_n$ is distal. 
\end{cor}

If $T$ is distal, then $\T$ is also distal and Theorem \ref{bar} (I) and Corollary \ref{cbar} holds for $T$. If $\T$ is distal, then for some $m\in\N$ and $l\in\Z$, $p^lT^m$ is distal. 

\subsection*{Acknowledgements.}
R.\ Shah would like to thank S.\,G.\ Dani and C.\,R.\,E.\ Raja for discussions on $p$-adic algebraic groups. R.\ Shah would like to acknowledge the support of DST (SERB) through the MATRICS 
research grant. A.\,K.\ Yadav would like to acknowledge the support for a research assistantship
from DST-PURSE grant in Jawaharlal Nehru University. We thank the referee for valuable comments and suggestions. 

\bibliographystyle{line}
\bibliography{JAMS-paper}

\begin{thebibliography}{99}

\bibitem{BJM}  J.\ F.\ Berglund, H.\ D.\ Junghenn and P.\ Milnes, `Analysis on semigroups. Function spaces, compactifications, representations'. 
Canadian Mathematical Society Series of Monographs and Advanced Texts. A Wiley-Interscience Publication. John Wiley \& Sons, Inc., New York, 1989. 

\bibitem{BoT} A.\ Borel and J.\ Tits, `Groupes r\`eductifs', {\em Publ. Math.\ Inst.\ Hautes Etudes Sci.} {\bf 27} (1965), 55--151. 

\bibitem{E4} R.\ Ellis,  `Distal transformation Groups', {\em Pacific J.\ Math.} {\bf 8} (1958) 401--405.

\bibitem{F6} H.\ Furstenberg, `The Structure of distal flows',  {\em Amer.\ J.\ Math.} {\bf 85} (1963), 477--515.

\bibitem{GR} Y.\ Guivarc'h, and C.\ R.\ E.\ Raja, `Recurrence and ergodicity of random walks on linear groups and on homogeneous spaces', {\em Ergodic Theory Dynam.\ Systems} 
{\bf 32} (2012), 1313--1349.

\bibitem{HeMo7} K.\ H.\ Hofmann and and P.\ S.\ Mostert, `Elements of compact semigroups'. Charles E.\ Merr ll Books, Inc., Columbus, Ohio (1966).

\bibitem{K} N.\ Koblitz, `$p$-adic numbers, $p$-adic Analysis, and zeta-functions', GTM 58, Springer-Verlag (1984).

\bibitem{M9} C.\ C.\ Moore, `Distal affine transformation groups', {\em Amer.\ J.\ Math.} {\bf 90} (1968) 733--751.

\bibitem{PR} V.\ Platonov and A.\ Rapinchuk, `Algebraic groups and number theory',
Translated from the 1991 Russian original by Rachel Rowen. Pure and
Applied Mathematics, 139. Academic Press, Inc., Boston, 1994.

\bibitem{RaSh10} C.\ R.\ E.\ Raja,  and R.\ Shah, `Distal actions and shifted convolution property', {\em Israel J.\ Math.} {\bf 177} (2010), 301--318.

\bibitem{RaSh11}  C.\ R.\ E.\ Raja and R.\ Shah, `Some properties of distal actions on locally compact groups', {\em Ergodic Theory Dynam.\ Systems} 2017, 
1--21. doi:10.1017/etds.2017.58.

\bibitem{Sh12} R.\ Shah,  `Orbits of distal actions on locally compact groups', {\em J.\ Lie Theory} {\bf 22} (2010), 586--599.

\bibitem{SY} R.\ Shah and A.\ K.\ Yadav, `Dynamics of distal actions on unit spheres', {\em Real Analysis Exchange}. To appear. 

\bibitem{W} J.\ S.\ P.\ Wang, `The Mautner phenomenon on $p$-adic Lie groups', {\em Math.\ Z.} {\bf 185} (1984) 403--411.

\end{thebibliography}

{\scriptsize
{\noindent Riddhi Shah \hspace{178pt}{\noindent Alok Kumar Yadav}\\  {School of Physical Sciences}\hspace{127pt}{School of Physical Sciences}\\ {Jawaharlal Nehru University (JNU)}\hspace{95pt}{Jawaharlal Nehru University (JNU)}\\ {New Delhi 110067, India}\hspace{136pt}{New Delhi 110067, India}\\ {\color{blue} rshah@jnu.ac.in}\hspace{169pt}{\color{blue} alokjnu90@gmail.com}\\ {\color{blue} riddhi.kausti@gmail.com}}

\end{document}